\documentclass[11pt]{article}
\usepackage{color,soul}
\usepackage{titlesec}
\usepackage{hyperref}
\usepackage{dirtytalk}
\usepackage{mathrsfs}
\usepackage{url}
\usepackage{lipsum}
\usepackage[thinlines]{easytable}
\usepackage{ntheorem}
\usepackage{mathtools}
\usepackage{graphicx}
\usepackage{amssymb}
\usepackage{calc}

\newlength{\depthofsumsign}
\usepackage{algorithm2e}
\RestyleAlgo{ruled}
\SetKwInput{KwData}{The data}
\SetKwInput{KwResult}{The result}
\usepackage{graphicx}
\usepackage{amsmath}
\usepackage{amssymb}
\usepackage{amsfonts}
\usepackage{array}
\usepackage{longtable}
\usepackage{tikz}
\usepackage{hhline}

\usepackage{multirow}
\usepackage{subfig}
\usepackage{hyperref}
\usepackage{mathtools}
\usepackage{tikz}
\usepackage{makecell}
\usepackage{pgfplots}
\usepackage{url}
\usepackage{afterpage}
\usepackage{fancyhdr}
\usepackage{setspace}
\usepackage{bm}
\DeclareMathAlphabet{\mathpzc}{OT1}{pzc}{m}{it}

\usepackage{diagbox}

\newtheorem{proposition} {Proposition}
\newtheorem*{proposition-non}{Proposition}

\newtheorem{remark}{Remark}

\usepackage{bm}

\allowdisplaybreaks

\newenvironment{proof}{\noindent{\bf Proof:}\indent}%
                      {\hfill $\Box$\par}

\newcommand{\sym}[1]{{\sf #1}}
\title{Counting unate and balanced monotone Boolean functions}
\author{Aniruddha Biswas and Palash Sarkar \\
Indian Statistical Institute \\
203, B.T.Road, Kolkata \\
India 700108. \\
Email: \{anib\_r, palash\}@isical.ac.in
}
\date{\today}

\begin{document}

\maketitle

\begin{abstract}
	We show that the problem of counting the number of $n$-variable unate functions reduces to the problem of counting the number of $n$-variable monotone functions.
	Using recently obtained results on $n$-variable monotone functions, we obtain counts of $n$-variable unate functions up to $n=9$. We use an enumeration strategy
	to obtain the number of $n$-variable balanced monotone functions up to $n=7$. We show that the problem of counting the number of $n$-variable balanced unate
	functions reduces to the problem of counting the number of $n$-variable balanced monotone functions, and consequently, we obtain the number of $n$-variable
	balanced unate functions up to $n=7$. Using enumeration, we obtain the numbers of equivalence classes of $n$-variable balanced monotone functions,  unate functions and balanced unate functions up to $n=6$. Further, for each of the considered sub-class
	of $n$-variable monotone and unate functions, we also obtain the corresponding numbers of $n$-variable non-degenerate functions. \\
	{\bf Keywords:} Boolean function, unate function, monotone function, Dedekind number, non-degenerate function, balanced functions, equivalence classes
	of functions. \\
	{\bf MSC:} 05A99.
\end{abstract}

\section{Introduction \label{sec-intro} }
For a positive integer $n$, an $n$-variable Boolean function $f$ is a map $f:\{0,1\}^n\rightarrow \{0,1\}$. A Boolean function $f$ is said to be monotone increasing
(resp. decreasing) in the $i$-th variable if 
$$f(x_1,\ldots,x_{i-1},0,x_{i+1},\ldots,x_n)\leq f(x_1,\ldots,x_{i-1},1,x_{i+1},\ldots,x_n)$$ 
$$(\mbox{resp. }f(x_1,\ldots,x_{i-1},0,x_{i+1},\ldots,x_n)\geq f(x_1,\ldots,x_{i-1},1,x_{i+1},\ldots,x_n))$$
for all possible $x_1,\ldots,x_{i-1},x_{i+1},\ldots,x_n\in\{0,1\}$. 
The function $f$ is said to be locally monotone or unate, if for each $i\in \{1,\ldots,n\}$, it is either monotone increasing or monotone decreasing in the $i$-th
variable. The function $f$ is said to be monotone increasing (or, simply monotone) if for each $i\in \{1,\ldots,n\}$, it is monotone increasing in the $i$-th variable.

Unate functions have been studied in the literature from various viewpoints such as switching 
theory~\cite{switch1_kohavi1970switching,switch2_mcnaughton1961unate,switch3_thayse1976logic,switch4_matheson1971recognition,switch5_betancourt1971derivation,switch6_pitchumani1988functional}, 
combinatorial aspects~\cite{switch3_thayse1976logic,comb1_feigelson1997forbidden,comp4_hata1991complexity}, and complexity theoretic 
aspects~\cite{comp1_balogh2022nearly,comp2_zwick19914n,comp3_morizumi2014sensitivity,comp4_hata1991complexity}. Monotone functions have been studied much 
more extensively than unate functions and have many applications so much so that it is difficult to mention only a few representative works. 

A Boolean function is degenerate on some variable if its output does not depend on the variable, and it is said to be non-degenerate if it is not degenerate on any variable. 
A Boolean function is said to be balanced if it takes the values 0 and 1 equal number of times. 
Two Boolean functions on the same number of variables are said to be equivalent if one can be obtained from the other by a permutation of variables. 
The notion of equivalence partitions the set of Boolean functions into equivalence classes. 

The number of $n$-variable monotone Boolean functions is called the $n$-th Dedekind number $D(n)$ after Dedekind who posed the problem in 1897. Till date, the $n$-th Dedekind
numbers has been obtained only up to $n=9$ (see~\cite{ListINtSequence,Dedekind1897103,Church,Ward,berman,Wiedemann19915,DBLP:journals/ipl/FidytekMSS01,jakel2023computation,vanhirtum2023computation}). 
A closed form summation formula for $D(n)$ was given in~\cite{Kisielewicz1988139}, though it was pointed out in~\cite{Korshunov2003929} that using the formula to
compute $D(n)$ has the same complexity as direct enumeration of all $n$-variable monotone Boolean functions. Dedekind numbers form the entry A000372 of~\cite{ListINtSequence}.
The number of $n$-variable non-degenerate Boolean functions can be obtained as the inverse binomial transform of the Dedekind numbers and are hence
also known up to $n=9$. These numbers form the entry A006126 of~\cite{ListINtSequence}.
The numbers of $n$-variable inequivalent monotone Boolean functions are known up to $n=9$ (see~\cite{STEPHEN201415,pawelski2021number,pawelski2023number}) and form
the entry A003182 of~\cite{ListINtSequence}. 


The focus of the present work is on counting unate and monotone Boolean functions under various restrictions. 
For $n\leq 5$, it is possible to enumerate all $n$-variable Boolean functions. Consequently, the problem of counting various sub-classes of $n$-variable
Boolean functions becomes a reasonably simple problem. Non-triviality of counting Boolean functions arises for $n\geq 6$. 

We show that the problem of counting unate functions reduces to the problem of counting monotone functions. Since the numbers of $n$-variable monotone functions
are known up to $n=9$, these values immediately provide the numbers of $n$-variable unate functions up to $n=9$.  
The problem of counting balanced monotone functions has not been considered in the literature. We use an enumeration strategy to count the number of balanced 
monotone functions up to $n=7$. We show that the problem of counting balanced unate functions reduces to the problem of counting balanced monotone functions.
Consequently, we obtain the numbers of $n$-variable balanced unate functions up to $n=7$. 
We further extend these results to obtain the numbers of non-degenerate balanced monotone functions, non-degenerate unate functions, and non-degenerate balanced unate functions. 

We describe a simple filtering technique for counting the number of equivalence classes of $n$-variable functions possessing a given property. Using this technique,
we compute the number of equivalence classes of $n$-variable balanced monotone functions.
Unlike the situation for counting functions, the problem of counting the number of equivalence classes of unate functions does not reduce to the
problem of counting the number of equivalence classes of monotone functions. So to count equivalence classes of unate functions, we used a method to generate
all $n$-variable unate functions and applied our filtering technique to obtain the number of equivalence classes of $n$-variable unate functions. This allowed us to
obtain the numbers of equivalence classes of $n$-variable unate and balanced unate functions up to $n=6$. 
We further extend these results to obtain the numbers of equivalence classes of $n$-variable non-degenerate monotone functions up to $n=9$. Moreover, we obtain the numbers of equivalence classes of $n$-variable balanced monotone functions, non-degenerate balanced monotone functions, non-degenerate unate functions and non-degenerate balanced unate functions up to $n=6$. 

To summarise, the new results that we obtain for monotone and unate functions are the following.
\begin{description}
	\item{\em Monotone:} 
		\begin{enumerate}
			\item Numbers of $n$-variable balanced monotone functions and $n$-variable non-degenerate balanced monotone functions up to $n=7$.
			\item Numbers of equivalence classes of $n$-variable non-degenerate monotone functions up to $n=9$. 
			\item Numbers of equivalence classes of $n$-variable balanced monotone functions, and $n$-variable non-degenerate balanced monotone functions up to $n=6$.
		\end{enumerate}
	\item{\em Unate:}
		\begin{enumerate}
			\item Numbers of $n$-variable unate functions and $n$-variable non-degenerate unate functions up to $n=9$.
			\item Numbers of $n$-variable balanced unate functions and $n$-variable non-degenerate balanced unate functions up to $n=7$.
			\item Numbers of equivalence classes of $n$-variable unate functions, $n$-variable non-degenerate unate functions, $n$-variable balanced unate
				functions, and $n$-variable non-degenerate balanced unate functions up to $n=6$.
		\end{enumerate}
\end{description}

\paragraph{Related counts:}
The number of NPN-equivalence classes\footnote{Two Boolean functions are said to be NPN equivalent, if one can be obtained from the
other by some combination of the following operations: a permutation of the variables, negation of a subset of the variables, and negation of the output.
Two functions are NPN inequivalent if they are not NPN equivalent.}
of unate Boolean functions has been studied (see~\cite{baugh72} and $A003183$ in~\cite{ListINtSequence}). 
A proper subclass of unate functions is the class of unate cascade functions which have been studied in~\cite{sasao1979,maitra1962cascaded,mukhopadhyay1969unate}.
Entry $A005612$ in~\cite{ListINtSequence} provides counts of unate cascade functions. 

\paragraph{Outline of the paper:}
In Section~\ref{sec-math} we describe the preliminaries and prove the mathematical results required to obtain the various counts. In Section~\ref{sec-fn-cnt} we address
the problem of counting various sub-classes of monotone and unate functions and in Section~\ref{sec-enum} we take up the problem of counting equivalence classes of
monotone and unate functions possessing a combination of several basic properties. Finally, Section~\ref{sec-conclu} provides the concluding remarks.

\section{Mathematical Results \label{sec-math} }
We fix some terminology and notation. The cardinality of a finite set $S$ will be denoted as $\#S$. 
For $x,y\in \{0,1\}$, $xy$ and $x\oplus y$ denote the AND and XOR operations respectively, and $\overline{x}$ denotes the complement (or negation) of $x$. 

Elements of $\{0,1\}^n$, $n\geq 2$, are $n$-bit strings (or vectors) and will be denoted using bold font. Given $n\geq 2$ and $1\leq i\leq n$, by $\mathbf{e}_i$
we will denote the $n$-bit string whose $i$-th bit is 1 and is 0 elsewhere.

Let $f$ be an $n$-variable Boolean function. 
The weight $\sym{wt}(f)$ of $f$ is the size of its support, i.e. $\sym{wt}(f)=\#\{\mathbf{x}: f(\mathbf{x})=1\}$.
An $n$-variable Boolean function $f$ can be uniquely represented by a binary string of length $2^n$ in the following manner: for 
$0\leq i<2^n$, the $i$-th bit of the string is the value of $f$ on the $n$-bit binary representation of $i$. We will use the same notation $f$ to 
denote the string representation of $f$. So $f_0\cdots f_{2^n-1}$ is the bit string of length $2^n$ which represents $f$.

By $\overline{f}$, we will denote the negation of $f$, i.e. $\overline{f}(\mathbf{x})=1$ if and only if $f(\mathbf{x})=0$. Let
$f^r$ be a Boolean function defined as $f^r(x_1,\ldots,x_n)=f(\overline{x}_1,\ldots,\overline{x}_n)$. The bit string representation of $f^r$ is the
reverse of the bit string representation of $f$. 

Let $g$ and $h$ be two $n$-variable Boolean functions having string representations $g_0\cdots g_{2^n-1}$ and $h_0\cdots h_{2^n-1}$. 
We write $g\leq h$ if $g_i\leq h_i$ for $i=0,\ldots,2^n-1$. From $g$ and $h$, it is possible to construct an $(n+1)$-variable function $f$ whose
string representation is obtained by concatenating the string representations of $g$ and $h$. We denote this construction as $f=g||h$. For
$(x_1,\ldots,x_{n+1})\in\{0,1\}^{n+1}$, we have
\begin{eqnarray}\label{eqn-concat}
	f(x_1,\ldots,x_{n+1}) & = & \overline{x}_1g(x_2,\ldots,x_{n+1}) \oplus x_1h(x_2,\ldots,x_{n+1}).
\end{eqnarray}

An $n$-variable Boolean function $f$ is said to be non-degenerate on the $i$-th variable, $1\leq i\leq n$, if there is an $\bm{\alpha}\in\{0,1\}^n$ such that
$f(\bm{\alpha})\neq f(\bm{\alpha}\oplus \mathbf{e}_i)$. The function $f$ is said to be non-degenerate, if it is non-degenerate on all the $n$ variables.

By a property $\mathcal{P}$ of Boolean functions, we will mean a subset of the set of all Boolean functions. For example, $\mathcal{P}$
could be the property of being balanced, being monotone, being unate, or a combination of these properties, where a combination of properties
is given by the intersection of the corresponding subsets of Boolean functions. For $n\geq 0$, let $\sym{P}_n$ denote the number of $n$-variable Boolean functions
possessing the property $\mathcal{P}$, and let $\sym{nd}\mbox{-}\sym{P}_n$ denote the number of $n$-variable non-degenerate Boolean functions possessing the property
$\mathcal{P}$. Since an $n$-variable function can be non-degenerate on $i$ variables for some $i\in \{0,\ldots,n\}$ and the $i$ variables can be chosen from the
$n$ variables in ${n\choose i}$ ways, we obtain the following result which shows that the sequence $\{\sym{P}_n\}_{n\geq 0}$ is given by the binomial transform of 
the sequence $\{\sym{nd}\mbox{-}\sym{P}_n\}_{n\geq 0}$.
\begin{proposition}\label{prop-nd}
	For any property $\mathcal{P}$ of Boolean functions,
	\begin{eqnarray}\label{eqn-nd}
		\sym{P}_n & = & \sum_{i=0}^n {n\choose i} \sym{nd}\mbox{-}\sym{P}_i.
	\end{eqnarray}
	Consequently,
	\begin{eqnarray}\label{eqn-nd-inv}
		\sym{nd}\mbox{-}\sym{P}_n & = & \sum_{i=0}^n (-1)^{n-i}{n\choose i} \sym{P}_i.
        \end{eqnarray}
\end{proposition}

\begin{remark}\label{rem-n=0}
	We assume that for $n=0$, there are two $n$-variable, non-degenerate, monotone (and hence unate), and unbalanced Boolean functions whose string representations are 
	0 and 1. 
\end{remark}

For $n\geq 0$, let $\sym{A}_n=2^{2^n}$ be the number of all $n$-variable Boolean functions, and let $\sym{B}_n={2^n\choose 2^{n-1}}$ be the number
of $n$-variable balanced Boolean functions. Let $\sym{nd}\mbox{-}\sym{A}_n$ be the number of all non-degenerate $n$-variable Boolean functions, and $\sym{nd}\mbox{-}\sym{B}_n$ 
be the number of all non-degenerate $n$-variable balanced Boolean functions. Using Proposition~\ref{prop-nd}, we obtain
\begin{eqnarray}\label{eqn-all-bal}
	\sym{nd}\mbox{-}\sym{A}_n = \sum_{i=0}^n (-1)^{n-i}{n\choose i} \cdot 2^{2^i} & \mbox{and} & 
	\sym{nd}\mbox{-}\sym{B}_n = \sum_{i=0}^n (-1)^{n-i}{n\choose i} \cdot {2^i\choose 2^{i-1}}.
\end{eqnarray}

For $n\geq 0$, by
$\sym{M}_n, \sym{BM}_n, \sym{U}_n$ and $\sym{BU}_n$, we will denote the numbers of $n$-variable monotone, balanced-monotone, unate, and balanced-unate functions
respectively, and by $\sym{nd}\mbox{-}\sym{M}_n, \sym{nd}\mbox{-}\sym{BM}_n, \sym{nd}\mbox{-}\sym{U}_n$ and $\sym{nd}\mbox{-}\sym{BU}_n$ we will denote the corresponding numbers
of non-degenerate functions. The relations between the number of $n$-variable functions possessing one of these properties and the number of non-degenerate
$n$-variable functions possessing the corresponding property are obtained from Proposition~\ref{prop-nd}. 

The following result relates the numbers of monotone and unate Boolean functions.
\begin{proposition}\label{prop-m-u}
For $n\geq 0$, the following holds.
	\begin{eqnarray}
		\sym{nd}\mbox{-}\sym{U}_n & = & 2^n \cdot \sym{nd}\mbox{-}\sym{M}_n, \label{eqn-m-u-nd} \\
		\sym{nd}\mbox{-}\sym{BU}_n & = & 2^n \cdot \sym{nd}\mbox{-}\sym{BM}_n, \label{eqn-m-u-b-nd} \\
		\sym{U}_n & \leq  & 2^n \cdot \sym{M}_n, \label{eqn-m-u} \\
		\sym{BU}_n & \leq & 2^n \cdot \sym{BM}_n,.\label{eqn-m-u-b}
	\end{eqnarray}
\end{proposition}
\begin{proof}
	First we consider~\eqref{eqn-m-u-nd} and~\eqref{eqn-m-u-b-nd}. We prove~\eqref{eqn-m-u-nd}, the proof of~\eqref{eqn-m-u-b-nd} being similar. 

	Let $f$ be an $n$-variable monotone Boolean function. Then it is easy to see that for any $\bm{\alpha}\in\{0,1\}^n$, the $n$-variable function
	$f_{\bm{\alpha}}$ is unate, where $f_{\bm{\alpha}}$ is defined as $f_{\bm{\alpha}}(\mathbf{x})=f(\mathbf{x}\oplus \bm{\alpha})$ for all $\mathbf{x}\in \{0,1\}^n$. 
	The proof of~\eqref{eqn-m-u-nd} follows from the following claim.

	\vspace{5pt}
	
	\noindent{\em Claim:} If $f$ is monotone, then the $2^n$ possible functions $f_{\bm{\alpha}}$ corresponding to the $2^n$ possible $\bm{\alpha}$'s are 
	distinct if and only if $f$ is non-degenerate. 
	
	\vspace{5pt}
	
	\noindent{\em Proof of the claim:} 
	Suppose $f$ is degenerate on the $i$-th variable. Then $f$ and $f_{\mathbf{e}_i}$ are equal. This proves one side of the claim. 
	So suppose that $f$ is non-degenerate. We have to show that for $\bm{\alpha}\neq \bm{\beta}$, $f_{\bm{\alpha}}$ and $f_{\bm{\beta}}$ are distinct functions. 
	Let if possible $f_{\bm{\alpha}}$ and $f_{\bm{\beta}}$ be equal.
	Note that since $f$ is non-degenerate, both $f_{\bm{\alpha}}$ and $f_{\bm{\beta}}$ are also non-degenerate. 
	Since $\bm{\alpha}=(\alpha_1,\ldots,\alpha_n)$ and $\bm{\beta}=(\beta_1,\ldots,\beta_n)$ are distinct, there is a $j$ in $\{1,\ldots,n\}$ such that $\alpha_j\neq \beta_j$. 
	Suppose without loss of generality that $\alpha_j=0$ and $\beta_j=1$. Since $f$ is monotone, it is monotone increasing in all variables and hence in the $j$-th
	variable. Further, since $\alpha_j=0$, $f_{\bm{\alpha}}$ is monotone increasing in the $j$-th variable and since $\beta_j=1$, $f_{\bm{\beta}}$ is monotone
	decreasing in the $j$-th variable. From $f_{\bm{\alpha}}$ is monotone increasing in the $j$-th variable we have that for all 
	$\mathbf{y}=(y_1,\ldots,y_n)\in \{0,1\}^n$ with $y_j=0$, $f_{\bm{\alpha}}(\mathbf{y})\leq f_{\bm{\alpha}}(\mathbf{y}\oplus \mathbf{e}_j)$. Further, since
	$f_{\bm{\alpha}}$ is non-degenerate, and hence non-degenerate on the $j$-th variable, equality cannot hold everywhere, i.e. there is a 
	$\mathbf{z}=(z_1,\ldots,z_n)\in \{0,1\}^n$ with $z_j=0$, such that 
	$f_{\bm{\alpha}}(\mathbf{z})=0$ and $f_{\bm{\alpha}}(\mathbf{z}\oplus \mathbf{e}_j)=1$. Since $f_{\bm{\alpha}}$ and $f_{\bm{\beta}}$ are assumed to be equal,
	it follows that $f_{\bm{\beta}}(\mathbf{z})=0$ and $f_{\bm{\beta}}(\mathbf{z}\oplus \mathbf{e}_j)=1$, which contradicts the fact that 
	$f_{\bm{\beta}}$ is monotone decreasing in the $j$-th variable. This proves the claim.

	Next we consider~\eqref{eqn-m-u} and~\eqref{eqn-m-u-b}. We provide the proof of~\eqref{eqn-m-u}, the proof of~\eqref{eqn-m-u-b} being similar.
The relation given by~\eqref{eqn-m-u} can be obtained from~\eqref{eqn-m-u-nd} and Proposition~\ref{prop-nd} using the following calculation.
\begin{eqnarray*}
	\sym{U}_n & = & \sum_{i=0}^{n} {n\choose i} \sym{nd}\mbox{-}\sym{U}_i 
	= \sum_{i=0}^{n} \left({n\choose i} \cdot 2^i \cdot \sym{nd}\mbox{-}\sym{M}_i\right) \leq 2^n \cdot \sum_{i=0}^{n} {n\choose i} \sym{nd}\mbox{-}\sym{M}_i 
	= 2^n\cdot \sym{M}_n.
\end{eqnarray*}
\end{proof}

We record two known facts about monotone functions. 
\begin{proposition}\label{fact_mono}$\cite{mono_gen}$
Let $g$ and $h$ be $n$-variable Boolean functions and $f=g||h$. Then $f$ is a monotone function if and only if $g$ and $h$ are both monotone functions and $g\leq h$.   
\end{proposition}
\begin{proposition}\label{fact_2}$\left(\mbox{A003183 of~\cite{ListINtSequence}}\right)$ 
	If $f$ is a monotone function then $\overline{f^{r}}$ is also a monotone function.
\end{proposition}

Next we present some results on unate and monotone functions which will be useful in our enumeration strategy. 
The first result is the analogue of Proposition~\ref{fact_mono} for unate functions.
\begin{proposition}\label{unate_proposition}
Let $g$ and $h$ be $n$-variable functions and $f=g||h$. Then $f$ is a unate function if and only if $g$ and $h$ are both unate functions satisfying the following
	two conditions.
	\begin{enumerate}
		\item For each variable, $g$ and $h$ are either both monotone increasing, or both monotone decreasing. 
		\item Either $g\leq h$ or $h\leq g$.      
	\end{enumerate}
\end{proposition}
\begin{proof}
	First consider the proof of the ``if'' part. Suppose $g$ and $h$ are unate functions satisfying the stated condition. We have to show that for each variable,
	$f$ is either monotone increasing, or monotone decreasing. Consider the variable $x_1$. If $g\leq h$, then from~\eqref{eqn-concat},
	$f$ is monotone increasing on $x_1$, while if $g\geq h$, then again from~\eqref{eqn-concat}, $f$ is monotone decreasing on $x_1$. Now consider
	any variable $x_i$, with $i\geq 2$. If $g$ and $h$ are both monotone increasing on $x_i$, then $f$ is also monotone increasing on $x_i$, while if $g$ and $h$ are 
	both monotone decreasing on $x_i$, then $f$ is also monotone decreasing on $x_i$. Since for each variable, $f$ is either monotone increasing, or
	monotone decreasing, it follows that $f$ is a unate function.

	For the converse, suppose that $f$ is a unate functions. Then for each variable $x_i$, $i\geq 1$, $f$ is either monotone increasing or monotone decreasing. 
	From~\eqref{eqn-concat}, it
	follows that for each variable $x_i$, $i\geq 2$, $g$ and $h$ are either both monotone increasing, or both monotone decreasing. So in particular, $g$ and $h$ are
	unate. If $f$ is monotone increasing for $x_1$, then $g\leq h$ and if $f$ is monotone decreasing for $x_1$, then $g\geq h$. 
\end{proof}

\begin{proposition}\label{prop2} 
	If $f$ is a unate function then $\overline{f}$ is also a unate function. 
\end{proposition}
\begin{proof}
	The proof is by induction on the number of variables $n$. The base case is $n=1$ and is trivial. Suppose the result holds for some $n\geq 1$. 
	Suppose that $f$ is an $(n+1)$-variable unate function. Then $f$ can be written as $f=g||h$, where $g$ and $h$ are $n$-variable unate functions satisfying the conditions
	in Proposition~\ref{unate_proposition}. Then $\overline{f}=\overline{g}||\overline{h}$. By induction hypothesis, $\overline{g}$ and $\overline{h}$
	are $n$-variable unate functions and the conditions in Proposition~\ref{unate_proposition} hold for $\overline{g}$ and $\overline{h}$. So $\overline{f}$ is
	a unate function.
\end{proof}

For $0\leq w\leq 2^n$, let $\sym{M}_{n,w}$ (resp. $\sym{U}_{n,w}$) be the number of $n$-variable monotone (resp. unate) Boolean functions of weight $w$.

\begin{proposition}\label{prop4}
	For any $n\geq 1$ and weight $w\in \{0,\ldots,2^n\}$,  $\sym{M}_{n,w}=\sym{M}_{n,2^n-w}$. 
\end{proposition}
\begin{proof}
	Proposition~\ref{fact_2} sets up a one-one correspondence between $n$-variable monotone functions having weight $w$ and $n$-variable monotone functions having weight $2^n-w$. 
	This shows that $\sym{M}_{n,w}=\sym{M}_{n,2^n-w}$.
\end{proof}

\begin{proposition}\label{prop3}
	For any $n\geq 1$ and weight $w\in \{0,\ldots,2^n\}$, $\sym{U}_{n,w}=\sym{U}_{n,2^n-w}$. 
\end{proposition}
\begin{proof}
	Proposition~\ref{prop2} sets up a one-one correspondence between $n$-variable unate functions having weight $w$ and $n$-variable unate functions having weight $2^n-w$. 
	This shows that $\sym{U}_{n,w}=\sym{U}_{n,2^n-w}$.
\end{proof}

\subsection{Equivalence \label{subsec-eq}}
Two Boolean functions are equivalent if they have the same number of variables and one can be obtained from the other by a 
permutation of variables. Let $\mathcal{P}$ be a property of Boolean functions. The set $\mathcal{P}$ is partitioned into equivalence classes by the notion of
equivalence. For $n\geq 0$, let $[P]_n$ denote the number of equivalence classes of $n$-variable functions possessing the property $\mathcal{P}$.
Also, let $\sym{nd}\mbox{-}[P]_n$ denote the number of equivalence classes of non-degenerate $n$-variable functions possessing the property $\mathcal{P}$.

\begin{remark}\label{rem-n=0=eq}
	We assume that for $n=0$, there are two equivalence classes of $n$-variable, non-degenerate, monotone (and hence unate), and unbalanced Boolean functions given 
	by $[0]$ and $[1]$. 
\end{remark}

We have the following analogue of Proposition~\ref{prop-nd}.
\begin{proposition}\label{prop-nd-eq}
	Let $\mathcal{P}$ be a property of Boolean functions which is closed under permutation of variables (i.e. if $f$ is in $\mathcal{P}$ and $g$ is obtained
	from $f$ by applying a permutation to the variables, then $g$ is also in $\mathcal{P}$). Then 
	\begin{eqnarray}\label{eqn-nd-eq}
		[P]_n & = & \sum_{i=0}^n \sym{nd}\mbox{-}[P]_i.
	\end{eqnarray}
	Consequently, 
	\begin{eqnarray}\label{eqn-nd-eq1}	
		\sym{nd}\mbox{-}[P]_n & = & [P]_n - [P]_{n-1}.
	\end{eqnarray}
\end{proposition}

For $n\geq 0$, let $[A]_n$ denote the number of equivalence classes of $n$-variable Boolean functions and $[B]_n$ denote the number of equivalence classes of $n$-variable
balanced Boolean functions. The values of $[A]_n$ and $[B]_n$ can be obtained using Polya's theory (see for example~\cite{RT09}). Let
$\sym{nd}\mbox{-}[A]_n$ denote the number of equivalence classes of $n$-variable non-degenerate Boolean functions and $\sym{nd}\mbox{-}[B]_n$ denote the
number of equivalence classes of $n$-variable non-degenerate balanced Boolean functions. Using Proposition~\ref{prop-nd-eq}, 
\begin{eqnarray}\label{eqn-all-bal-eq}
	\sym{nd}\mbox{-}[A]_n=[A]_n - [A]_{n-1} & \mbox{and} & \sym{nd}\mbox{-}[B]_n=[B]_n - [B]_{n-1}.
\end{eqnarray}

For $n\geq 0$, by $[M]_n$, $[BM]_n$, $[U]_n$ and $[BU]_n$ we will denote the numbers of equivalence classes of $n$-variable monotone, balanced-monotone, unate, and 
balanced-unate functions respectively and by 
$\sym{nd}\mbox{-}[M]_n$, $\sym{nd}\mbox{-}[BM]_n$, $\sym{nd}\mbox{-}[U]_n$ and $\sym{nd}\mbox{-}[BU]_n$ 
we will denote the corresponding numbers of equivalence classes of non-degenerate functions. 
The following result is the analogue of Proposition~\ref{prop-m-u}. 
\begin{proposition}\label{prop-m-u-eq}
For $n\geq 0$, the following holds.
	\begin{eqnarray}
		\sym{nd}\mbox{-}[U]_n & \leq & 2^n \cdot \sym{nd}\mbox{-}[M]_n, \label{eqn-m-u-nd-eq} \\
		\sym{nd}\mbox{-}[BU]_n & \leq & 2^n \cdot \sym{nd}\mbox{-}[BM]_n, \label{eqn-m-u-nd-b-eq} \\
		\mbox{$[U]$}_n & \leq & 2^n \cdot [M]_n, \label{eqn-m-u-eq} \\
		\mbox{$[BU]$}_n & \leq & 2^n \cdot [BM]_n, \label{eqn-m-u-b-eq} 
	\end{eqnarray}
\end{proposition}
The relations given by~\eqref{eqn-m-u-eq} and~\eqref{eqn-m-u-b-eq} are analogues of~\eqref{eqn-m-u} and~\eqref{eqn-m-u-b} respectively. 
However, unlike~\eqref{eqn-m-u-nd} and~\eqref{eqn-m-u-b-nd}, we do not have equalities in~\eqref{eqn-m-u-nd-eq} and~\eqref{eqn-m-u-nd-b-eq}.
The reason is that two distinct input translations of a non-degenerate monotone function can lead to two unate functions which are equivalent.
An example is the following. Suppose $f(X_1,X_2)=X_1X_2$, i.e. $f$ is the AND function. Let $g(X_1,X_2)=f(1\oplus X_1, X_2)=(1\oplus X_1)X_2$ and 
$h(X_1,X_2)=f(X_1,1\oplus X_2)=X_1(1\oplus X_2)$. Then $g(X_1,X_2)=h(X_2,X_1)$, i.e. $g$ and $h$ are distinct, but equivalent unate functions obtained 
by distinct input translations from the monotone function $f$.

\section{Counting Functions \label{sec-fn-cnt}}
In this section, we consider the problem of counting various sub-classes of monotone and unate Boolean functions.

\subsection{Monotone Functions \label{subsec-fn-cnt-monotone}}
Note that $\sym{M}_n$ is the $n$-th Dedekind number. For $0\leq n\leq 9$, the
values of $\sym{M}_n$ are known~\cite{ListINtSequence}, with the value of $\sym{M}_9$ being obtained recently and independently by two groups of 
researchers~\cite{vanhirtum2023computation,jakel2023computation}.
The values of $\sym{M}_n$ form entry A000372 of~\cite{ListINtSequence}. The numbers of non-degenerate $n$-variable monotone functions, $\sym{nd}\mbox{-}\sym{M}_n$,
form entry A006126 of~\cite{ListINtSequence}. 

We used enumeration to obtain the number $\sym{BM}_n$ of $n$-variable balanced monotone functions. For $n\leq 6$, we enumerated all monotone functions and
counted only the balanced functions. Our strategy for enumerating monotone functions is based on Proposition~\ref{fact_mono}. The approach is the following.
First generate all $1$-variable monotone functions and store these. For $n\geq 2$, to generate all $n$-variable monotone functions, we consider each 
pair $(g,h)$ of $(n-1)$-variable monotone functions and check whether the pair satisfies the condition of Proposition~\ref{fact_mono}. If it does, then $f=g||h$ is
stored. To generate all $n$-variable monotone functions, this approach requires considering $(\sym{M}_{n-1})^2$ pairs. The enumeration and filtering out unbalanced functions
allows us to obtain the values of $\sym{BM}_n$, for $n=1,\ldots,6$.

\begin{remark}\label{rem-nd-not}
	To obtain a faster method, one may consider generating only non-degenerate functions using Proposition~\ref{fact_mono}. This, however, does not work. It is indeed
	true that if $g$ and $h$ are distinct non-degenerate functions, $f=g||h$ is also non-degenerate. On the other hand, it is possible that one of $g$ or $h$ is degenerate,
	but $f$ is non-degenerate. For example, take $g$ to be the 2-variable constant function whose string representation is $0000$, and $h$ to be the 2-variable AND function
	whose string representation is given by $0001$. Then the string representation of the 3-variable function $f=g||h$ is $00000001$ which is the AND of the three
	variables and hence non-degenerate. So the set of all non-degenerate $n$-variable monotone functions cannot be obtained by concatenating only non-degenerate
	$(n-1)$-variable monotone functions.
\end{remark}

To obtain $\sym{BM}_7$ we used a faster method. After enumerating all 6-variable monotone functions, we divided these functions into groups where all functions in the
same group have the same weight. Our modified strategy is to take two $n$-variable monotone functions $g$ and $h$, where $g$ has weight $w$ and $h$ has weight $2^n-w$ 
and check whether $g\leq h$. If the check passes, then we generate the $(n+1)$-variable balanced monotone function $f=g||h$. 
Recall that for $0\leq w\leq 2^n$, there are $\sym{M}_{n,w}$ $n$-variable monotone functions having weight $w$. The number of pairs needed to be considered by the
modified method is
$$
\sum_{w=0}^{2^n} \sym{M}_{n,w}\sym{M}_{n,2^n-w} = \sum_{w=0}^{2^n} \left(\sym{M}_{n,w} \right)^2,
$$
where the equality follows from Proposition~\ref{prop4}. Substituting $n=6$ and using the values of $\sym{M}_{6,w}$ obtained through enumeration, we find that the 
modified strategy for generating 7-variable balanced monotone
functions requires considering $\sum_{w=0}^{64}\left(\sym{M}_{6,w} \right)^2\approx 2^{40}$ pairs, while the previous strategy would have required considering
$(\sym{M}_6)^2\approx 2^{45}$ pairs. 

\begin{remark}\label{rem-once}
Note that the above procedure to generate balanced monotone functions can be applied only once. It uses the set of all
$n$-variable monotone functions to generate the set of all $(n+1)$-variable balanced monotone functions. Since this does not provide all $(n+1)$-variable monotone
functions, it cannot be applied to generate the set of all $(n+2)$-variable balanced monotone functions. 
\end{remark}

Having obtained $\sym{BM}_n$, for $n=1,\ldots,7$, we use Proposition~\ref{prop-nd} to obtain the values of $\sym{nd}\mbox{-}\sym{BM}_n$, i.e. the number of 
$n$-variable non-degenerate balanced monotone functions. The obtained values of $\sym{BM}_n$ and $\sym{nd}\mbox{-}\sym{BM}_n$ are given in Table~\ref{tab-b-m}.

\begin{table}
\centering
	{\scriptsize
	\begin{tabular}{|c|r|r|}
		\hline
		$n$ & \multicolumn{1}{c|}{$\sym{BM}_n$} & \multicolumn{1}{c|}{$\sym{nd}\mbox{-}\sym{BM}_n$} \\ \hline
		0 & 0 & 0\\ \hline
		1 & 1 & 1\\ \hline
		2 & 2 & 0\\ \hline
		3 & 4 & 1\\ \hline
		4 & 24 & 16\\ \hline
		5 & 621 & 526  \\ \hline
		6 & 492288 & 488866  \\ \hline
		7 & 81203064840 & 81199631130 \\ \hline
	\end{tabular}
	\caption{Numbers of $n$-variable balanced monotone and non-degenerate balanced monotone functions for $0\leq n\leq 7$. \label{tab-b-m}}
	}
\end{table}

\subsection{Unate Functions \label{subsec-unate-cnt}}
The problem of counting unate functions reduces to the problem of counting monotone functions in the following manner. Suppose we wish to obtain the number
$\sym{U}_n$ of $n$-variable unate functions. Using Proposition~\ref{prop-nd}, this reduces to the problem of obtaining $\sym{nd}\mbox{-}\sym{U}_i$, for 
$0\leq i\leq n$. From~\eqref{eqn-m-u-nd}, this reduces to the problem of obtaining $\sym{nd}\mbox{-}\sym{M}_i$ for $0\leq i\leq n$. Using another application
of Proposition~\ref{prop-nd} reduces the problem of obtaining $\sym{nd}\mbox{-}\sym{M}_i$ to that of obtaining $\sym{M}_j$ for $0\leq j\leq i$. So to obtain
$\sym{U}_n$, it is sufficient to know $\sym{M}_i$ for $0\leq i\leq n$. Since the values of $\sym{M}_i$ are known for $0\leq i\leq 9$, we can obtain the
values of $\sym{U}_n$ for $0\leq n\leq 9$. From these values, using Proposition~\ref{prop-nd}, we obtain the values of $\sym{nd}\mbox{-}\sym{U}_n$ for
$0\leq n\leq 9$. The values of $\sym{U}_n$ and $\sym{nd}\mbox{-}\sym{U}_n$ are shown in Table~\ref{tab-unate}.

\begin{table}
\centering
	{\scriptsize
	\begin{tabular}{|c|r|r|}
		\hline
		$n$ & \multicolumn{1}{c|}{$\sym{U}_n$} & \multicolumn{1}{c|}{$\sym{nd}\mbox{-}\sym{U}_n$} \\ \hline
		0 &2 & 2\\ \hline
		1 & 4& 2\\ \hline
		2 &14 & 8\\ \hline
		3 & 104 & 72\\ \hline
		4 &2170 & 1824\\ \hline
		5 &230540 & 220608\\ \hline
		6 & 499596550 & 498243968\\ \hline
		7 & 309075799150640& 309072306743552\\ \hline
		8 &14369391928071394429416818 &14369391925598802012151296 \\ \hline
		9 & 146629927766168786368451678290041110762316052 & 146629927766168786239127150948525247729660416\\ \hline
	\end{tabular}
	\caption{Numbers of $n$-variable unate and non-degenerate unate functions for $0\leq n\leq 9$. \label{tab-unate}}
	}
\end{table}

In a similar manner, using Proposition~\ref{prop-nd} and~\eqref{eqn-m-u-b-nd}, the problem of counting balanced unate functions can be reduced to the problem
of counting balanced monotone functions. Since we have obtained the values of $\sym{BM}_i$ for $0\leq i\leq 7$, we obtain the values of 
$\sym{BU}_n$ for $0\leq n\leq 7$. Using Proposition~\ref{prop-nd}, this gives us the values of $\sym{nd}\mbox{-}\sym{BU}_n$ for $0\leq n\leq 7$. 
The values of $\sym{BU}_n$ and $\sym{nd}\mbox{-}\sym{BU}_n$ are shown in Table~\ref{tab-bal-unate}.

\begin{table}
\centering
	{\scriptsize
	\begin{tabular}{|c|r|r|}
		\hline
		$n$ & \multicolumn{1}{c|}{$\sym{BU}_n$} & \multicolumn{1}{c|}{$\sym{nd}\mbox{-}\sym{BU}_n$} \\ \hline
		0 &0 &0 \\ \hline
		1 & 2 & 2 \\ \hline
		2 &4 &0 \\ \hline
		3 &14 &8 \\ \hline
		4 &296 &256 \\ \hline
		5 &18202 & 16832\\ \hline
		6 & 31392428& 31287424\\ \hline
		7 & 10393772159334 &10393552784640 \\ \hline
	\end{tabular}
	\caption{Numbers of $n$-variable balanced unate and non-degenerate balanced unate functions for $0\leq n\leq 7$. \label{tab-bal-unate}}
	}
\end{table}


\section{Counting Equivalence Classes of Functions \label{sec-enum} }
In this section, we present the results on numbers of equivalence classes of various subsets of monotone and unate functions.

\subsection{Filtering Procedure \label{subsec-filter}}
The basic problem of enumerating equivalence classes is the following. 
Let $\mathcal{S}$ be a subset of the set of all $n$-variable Boolean functions. Given $\mathcal{S}$, we wish to generate a set $\mathcal{T}\subseteq \mathcal{S}$ of functions
such that no two functions in $\mathcal{T}$ are equivalent, and each function in $\mathcal{S}$ is equivalent to some function in $\mathcal{T}$. The technique for such
filtering is the following. 

Given a permutation $\pi$ of $\{1,\ldots,n\}$, we define a permutation $\pi^{\star}$ of $\{0,\ldots,2^n-1\}$ as follows. For $i\in \{0,\ldots,2^n-1\}$, let 
$(i_1,\ldots,i_n)$ be the $n$-bit binary representation of $i$. Then $\pi^{\star}(i)=j$, where the $n$-bit binary representation of $j$ is
$(i_{\pi(1)},\ldots,i_{\pi(n)})$. 
Given an $n$-variable function $f$, let $f^\pi$ denote the function such that for all $(x_1,\ldots,x_n)\in\{0,1\}^n$, $f^\pi(x_1,\ldots,x_n)=f(x_{\pi(1)},\ldots,x_{\pi(n)})$.
Suppose $f_0\cdots f_{2^n-1}$ is the bit string representation of $f$. Then the bit string representation of $f^\pi$ is
$f_{\pi^\star(0)}\cdots f_{\pi^\star(2^n-1)}$. 

Note that for each permutation $\pi$, the permutation $\pi^\star$ can be pre-computed and stored as an array say $P[0,\ldots,2^n-1]$. Suppose the bit string
representation of $f$ is stored as an array $A[0,\ldots,2^n-1]$. Then the bit string representation of $f^\pi$ is obtained as the array $B[0,\ldots,2^n-1]$, where
$B[i]=A[P[i]]$, for $i=0,\ldots,2^n-1$. So obtaining $f^{\pi}$ becomes simply a matter of array reindexing. 

Consider the set of functions $\mathcal{S}$ to be filtered is given as a list of string representations of the functions. We incrementally generate $\mathcal{T}$ as 
follows. The first function
in $\mathcal{S}$ is moved to $\mathcal{T}$. We iterate over the other functions in $\mathcal{S}$. For a function $f$ in $\mathcal{S}$, we generate
$f^\pi$ for all permutations $\pi$ of $\{1,\ldots,n\}$ using the technique described above. For each such $f^\pi$, we check whether it is present in $\mathcal{T}$.
If none of the $f^{\pi}$'s are present in $\mathcal{T}$, then we append $f$ to $\mathcal{T}$. At the end of the procedure, $\mathcal{T}$ is the desired set of functions.

The check for the presence of $f^\pi$ in $\mathcal{T}$ involves a search in $\mathcal{T}$. This is done using binary search. To apply binary search on a list, it 
is required that the list be sorted. To ensure this, we initially ensure that $\mathcal{S}$ is sorted (either by generating it in a sorted manner, or by sorting
it after generation). This ensures that at any point of time, $\mathcal{T}$ is also a sorted list, so that binary search can be applied. 

\subsection{Monotone \label{subsec-monotone-eq}}
For $n\geq 0$, the numbers $[M]_n$ of equivalence classes of $n$-variable monotone functions form entry A003182 of~\cite{ListINtSequence}. Using Proposition~\ref{prop-nd-eq}, it
is possible to find the numbers $\sym{nd}\mbox{-}[M]_n$ of equivalence classes of $n$-variable monotone functions. These values are shown in Table~\ref{tab-m-eq}.

\begin{table}
\centering
	{\scriptsize
	\begin{tabular}{|c|r|}
		\hline
		$n$ & \multicolumn{1}{c|}{$\sym{nd}\mbox{-}[M]_n$} \\ \hline
		0 & 2\\ \hline
		1 & 1\\ \hline
		2 & 2\\ \hline
		3 & 5\\ \hline
		4 & 20\\ \hline
		5 & 180\\ \hline
		6 & 16143\\ \hline
		7 & 489996795\\ \hline
		8 & 1392195548399980210\\ \hline
		9 & 789204635842035039135545297410259322\\ \hline
	\end{tabular}
	\caption{Numbers of equivalence classes of $n$-variable non-degenerate monotone functions for $0\leq n\leq 9$. \label{tab-m-eq}}
	}
\end{table}

For $0\leq n\leq 6$, the numbers $[BM]_n$ of equivalence classes of $n$-variable balanced monotone functions are obtained by applying the filtering
procedure described in Section~\ref{subsec-filter} to the strategy for generating balanced monotone functions described in Section~\ref{subsec-fn-cnt-monotone}.
Next applying Proposition~\ref{prop-nd-eq}, we obtained the numbers $\sym{nd}\mbox{-}[BM]_n$ of equivalence classes of $n$-variable non-degenerate balanced monotone
functions. The values of $[BM]_n$ and $\sym{nd}\mbox{-}[BM]_n$ are shown in Table~\ref{tab-b-m-eq}.

We briefly consider the computation required to obtain $[BM]_{7}$.
From Table~\ref{tab-b-m}, $\sym{BM}_{7}\allowbreak =\allowbreak 81203064840\allowbreak \approx\allowbreak  2^{36.24}$. For each 7-variable balanced monotone function $f$, 
it is required to consider $7!=5040\approx 2^{12.29}$ functions $f^{\pi}$ for the $7!$ permutations $\pi$ of $\{1,\ldots,7\}$. So a total
of about $2^{48.53}$ functions would have to be considered. For each of these functions a binary search is required on the partially generated
set of functions $\mathcal{T}$ and requires performing about $\log_2\#\mathcal{T}$ comparisons. So the total number of comparisons required is somewhat above $2^{50}$. 
This amount of computation is not presently feasible on the computing resources available to us.

\begin{table}
\centering
	{\scriptsize
	\begin{tabular}{|c|r|r|}
		\hline
		$n$ & \multicolumn{1}{c|}{$[BM]_n$} & \multicolumn{1}{c|}{$\sym{nd}\mbox{-}[BM]_n$} \\ \hline
		0 &0 & 0\\ \hline
		1 & 1& 1\\ \hline
		2 & 1& 0\\ \hline
		3 & 2& 1\\ \hline
		4 & 4& 2\\ \hline
		5 & 16&12 \\ \hline
		6 & 951&935 \\ \hline
	\end{tabular}
	\caption{Numbers of equivalence classes of $n$-variable balanced monotone and non-degenerate balanced monotone functions for $0\leq n\leq 6$. \label{tab-b-m-eq}}
	}
\end{table}

\subsection{Unate \label{subsec-unate-eq}}


In the case of counting functions, the problems of counting unate and balanced unate functions reduce to the problems of counting monotone and balanced monotone
functions respectively. In the case of counting equivalence classes of functions, such reduction is no longer possible (using the results that we could prove).
	The reason is that unlike~\eqref{eqn-m-u-nd} which expresses the number of non-degenerate unate functions in terms of the number of non-degenerate
	monotone function, the relation~\eqref{eqn-m-u-nd-eq} only provides an upper bound on the number of equivalence classes of non-degenerate unate
	functions in terms of the number of equivalence classes of non-degenerate monotone functions. 
	
In view of the above, for counting equivalence classes of unate functions, we resorted to the technique of enumerating unate functions and then using the
	technique described in Section~\ref{subsec-filter} to obtain the number of equivalence classes. 

	The technique of generating all unate functions is based on Proposition~\ref{unate_proposition}. Along with the string representation of a unate
	function, we also need to record whether the function is increasing or decreasing in each of its variables. This is recorded as the signature of the function.
	The special cases of the two constant functions cause some complications in the definition of the signature. 

For an $n$-variable unate function $f$, we define its signature, denoted $\sym{sig}(f)$, to be an element of $\{0,1\}^n\cup \{\mathfrak{z},\mathfrak{o}\}$ in the 
following manner. If $f$ is the constant function $1$, then $\sym{sig}(f)=\mathfrak{o}$, if $f$ is the constant function $0$, then 
$\sym{sig}(f)=\mathfrak{z}$; otherwise $\sym{sig}(f)$ is an $n$-bit string $\alpha$, where for $i=1,\ldots,n$, 
$\alpha_i=1$ if $f$ is monotone increasing in the variable $x_i$, and $\alpha_i=0$ if $f$ is monotone decreasing in the variable $x_i$. 
The signature $\sym{sig}(f)$ encodes whether $f$ is monotone increasing or monotone decreasing on each variable. The function $f$ is both monotone increasing
and monotone decreasing in all the variables if and only if it is a constant function. The signatures of the constant functions are defined appropriately.

For enumeration, the bit string representation of the functions are used. A unate function and its signature are stored as a pair. 
Consider the following recursive algorithm to generate all $n$-variable unate functions and their signatures for $n\geq 1$. 
At the base step, i.e. for $n=1$, store the four pairs of 1-variable unate functions and their signatures as $(00,\mathfrak{z})$, 
$(01,1)$, $(10,0)$ and $(11,\mathfrak{o})$. Suppose that for some $n\geq 1$, we have already generated all $n$-variable
unate functions and their signatures. The generation of all $(n+1)$-variable unate functions and their signatures are done as follows. For any two
function-signature pairs $(g,\sym{sig}(g))$ and $(h,\sym{sig}(h))$, where $g$ and $h$ are $n$-variable unate functions (which are not necessarily distinct), perform the 
following checks: 
\begin{enumerate}
	\item Whether at least one of $\sym{sig}(g)$ or $\sym{sig}(h)$ is equal to either $\mathfrak{z}$ or $\mathfrak{o}$ (i.e. whether at least one of
		$g$ or $h$ is a constant function).
	\item $\sym{sig}(g)=\sym{sig}(h)=\alpha$, and either $g\leq h$ or $h\leq g$ holds. 
\end{enumerate}
If either of the checks pass, then generate $f=g||h$, and determine $\sym{sig}(f)$ as follows.
\begin{eqnarray}\label{eqn-sig-f}
	\sym{sig}(f) 
	& = & \left\{
	\begin{array}{ll}
		\mathfrak{z} & \mbox{if } \sym{sig}(g)=\sym{sig}(h)=\mathfrak{z}, \\
		\mathfrak{o} & \mbox{if } \sym{sig}(g)=\sym{sig}(h)=\mathfrak{o}, \\
		1^{n+1}   & \mbox{if } \sym{sig}(g)=\mathfrak{z},\  \sym{sig}(h)=\mathfrak{o}, \\
		0^{n+1}   & \mbox{if } \sym{sig}(g)=\mathfrak{o},\  \sym{sig}(h)=\mathfrak{z}, \\
		1||\alpha & \mbox{if } \sym{sig}(g)=\mathfrak{z},\  \sym{sig}(h)=\alpha\in \{0,1\}^n, \\
		0||\alpha & \mbox{if } \sym{sig}(g)=\mathfrak{o},\  \sym{sig}(h)=\alpha\in \{0,1\}^n \\
		1||\alpha & \mbox{if } \sym{sig}(g)=\alpha\in \{0,1\}^n,\  \sym{sig}(h)=\mathfrak{o}, \\
		0||\alpha & \mbox{if } \sym{sig}(g)=\alpha\in \{0,1\}^n,\  \sym{sig}(h)=\mathfrak{z}, \\
		1||\alpha & \mbox{if } g\leq h,\  \sym{sig}(g)=\sym{sig}(h)=\alpha \in \{0,1\}^n, \\
		0||\alpha & \mbox{if } g\geq h,\  \sym{sig}(g)=\sym{sig}(h)=\alpha \in \{0,1\}^n.
	\end{array}
	\right.
\end{eqnarray}
Store $(f,\sym{sig}(f))$. Proposition~\ref{unate_proposition} assures us that this recursive procedure generates all $(n+1)$-variable unate functions and their signatures. 

To generate all $(n+1)$-variable unate functions, the above method requires considering all pairs of $n$-variable unate functions, i.e. a total of
$\left(U(n)\right)^2$ options. Applying the fitering strategy of Section~\ref{subsec-filter} we obtain the value of 
$[U]_n$. Next using Proposition~\ref{prop-nd-eq} we obtain the value of $\sym{nd}\mbox{-}[U]_n$. 
We could perform this computation for $n\leq 6$. The obtained values of $[U]_n$ and $\sym{nd}\mbox{-}[U]_n$ are shown in Table~\ref{tab-u-eq}.
To generate all 7-variable unate functions using this option requires considering $\left(U(6)\right)^2\approx 2^{57.8}$ pairs of functions. This is not feasible on the
computing facility available to us.

\begin{table}
\centering
	{\scriptsize
	\begin{tabular}{|c|r|r|}
		\hline
		$n$ & \multicolumn{1}{c|}{$[U]_n$} & \multicolumn{1}{c|}{$\sym{nd}\mbox{-}[U]_n$} \\ \hline
		0 &2 & 2\\ \hline
		1 & 4& 2\\ \hline
		2 &10 &6 \\ \hline
		3 & 34& 24\\ \hline
		4 & 200&166 \\ \hline
		5 &3466 &3266 \\ \hline
		6 & 829774& 826308\\ \hline
	\end{tabular}
	\caption{Numbers of equivalence classes of $n$-variable unate and non-degenerate unate functions for $0\leq n\leq 6$. \label{tab-u-eq}}
	}
\end{table}

To obtain the set of $n$-variable balanced unate functions, after generating the set of all $n$-variable unate functions, we remove the ones that are unbalanced. Then to the
resulting set, we apply the technique of Section~\ref{subsec-filter} to obtain the number $[BU]_n$ of equivalence classes of $n$-variable balanced unate functions. 
Subsequently, we apply Proposition~\ref{prop-nd-eq} to obtain the number $\sym{nd}\mbox{-}[BU]_n$ of equivalence classes of $n$-variable non-degenerate balanced
unate functions. The values of $[BU]_n$ and $\sym{nd}\mbox{-}[BU]_n$ are shown in Table~\ref{tab-u-b-eq}

\begin{table}
\centering
	{\scriptsize
	\begin{tabular}{|c|r|r|}
		\hline
		$n$ & \multicolumn{1}{c|}{$[BU]_n$} & \multicolumn{1}{c|}{$\sym{nd}\mbox{-}[BU]_n$} \\ \hline
		0 & 0& 0\\ \hline
		1 & 2& 2\\ \hline
		2 & 2 &0 \\ \hline
		3 & 6& 4\\ \hline
		4 & 24& 18 \\ \hline
		5 & 254& 230\\ \hline
		6 & 50172&49918 \\ \hline
	\end{tabular}
	\caption{Numbers of equivalence classes of $n$-variable balanced unate and non-degenerate balanced unate functions for $0\leq n\leq 6$. \label{tab-u-b-eq}}
	}
\end{table}


\section{Concluding Remarks \label{sec-conclu}}
We have obtained the numbers of $n$-variable unate and monotone functions possessing a combination of some basic properties. Also, we have obtained the numbers
of equivalence classes of $n$-variable unate and monotone functions also possessing a combination of those same properties. Our work raises a number of questions that
may be pursued in the future. One such question is whether the techniques for counting monotone functions from the recent works~\cite{jakel2023computation,vanhirtum2023computation} can be applied to the
problem of counting balanced monotone functions. Another similar question is whether the techniques for counting the number of equivalence classes of monotone functions
from~\cite{pawelski2021number,pawelski2023number} can be applied to the problem of counting the number of equivalence classes of balanced monotone functions. A third question is whether the
techniques for counting the number of equivalence classes of monotone functions from~\cite{pawelski2021number,pawelski2023number} can be applied to the problem of counting the number of equivalence
classes of unate functions. Positive answers to these questions will allow extending the results that we could obtain up to $n=6$ or $n=7$ to $n=9$. 

\section*{Acknowledgement} We are grateful to an anonymous reviewer of the PhD thesis of the first author for suggesting that the number of $n$-variable non-degenerate unate
functions is equal to $2^n$ times the number of $n$-variable non-degenerate monotone functions. 


\begin{thebibliography}{10}

\bibitem{mono_gen}
Valentin Bakoev.
\newblock Generating and identification of monotone {B}oolean functions.
\newblock In {\em {M}athematics and {E}ducation in {M}athematics, {S}ofia},
  pages 226--232, 2003.

\bibitem{comp1_balogh2022nearly}
J{\'o}zsef Balogh, Dingding Dong, Bernard Lidick{\`y}, Nitya Mani, and Yufei
  Zhao.
\newblock Nearly all $k$-{SAT} functions are unate.
\newblock \url{https://arxiv.org/pdf/2209.04894.pdf}, 2022.

\bibitem{baugh72}
Charles~R. Baugh.
\newblock Generation of representative functions of the {NPN} equivalence
  classes of unate {B}oolean functions.
\newblock {\em {IEEE} Transactions on Computers}, 21(12), 1972.

\bibitem{berman}
Joel Berman and Peter Köhler.
\newblock Cardinalities of finite distributive lattices.
\newblock {\em Mitt. Math. Sem. Giessen}, 121:103--124, 1976.

\bibitem{switch5_betancourt1971derivation}
Rodolfo Betancourt.
\newblock Derivation of minimum test sets for unate logical circuits.
\newblock {\em IEEE Transactions on Computers}, 100(11):1264--1269, 1971.

\bibitem{Church}
Randolph Church.
\newblock {Nunmerical analysis of certain free distributive structures}.
\newblock {\em Duke Mathematical Journal}, 6(3):732 -- 734, 1940.

\bibitem{Dedekind1897103}
Richard Dedekind.
\newblock Über zerlegungen von zahlen durch ihre grössten gemeinsamen teiler.
\newblock {\em Gesammelte Werke}, 2:103 – 148, 1897.

\bibitem{comb1_feigelson1997forbidden}
Aaron Feigelson and Lisa Hellerstein.
\newblock The forbidden projections of unate functions.
\newblock {\em Discrete Applied Mathematics}, 77(3):221--236, 1997.

\bibitem{DBLP:journals/ipl/FidytekMSS01}
Robert Fidytek, Andrzej~W. Mostowski, Rafal Somla, and Andrzej Szepietowski.
\newblock Algorithms counting monotone {B}oolean functions.
\newblock {\em Inf. Process. Lett.}, 79(5):203--209, 2001.

\bibitem{comp4_hata1991complexity}
Yutaka Hata, Masaharu Yuhara, Fujio Miyawaki, and Kazuharu Yamato.
\newblock On the complexity of enumerations for multiple-valued {K}leenean
  functions and unate functions.
\newblock In {\em 1991 Proceedings of the Twenty-First International Symposium
  on Multiple-Valued Logic}, pages 55--56. IEEE Computer Society, 1991.

\bibitem{vanhirtum2023computation}
Lennart~Van Hirtum, Patrick~De Causmaecker, Jens Goemaere, Tobias Kenter,
  Heinrich Riebler, Michael Lass, and Christian Plessl.
\newblock A computation of {$D(9)$} using {FPGA} supercomputing.
\newblock \url{https://arxiv.org/abs/2304.03039}, 2023.

\bibitem{jakel2023computation}
Christian J\"{a}kel.
\newblock A computation of the ninth {D}edekind number.
\newblock \url{https://arxiv.org/abs/2304.00895}, 2023.

\bibitem{Kisielewicz1988139}
Andrzej Kisielewicz.
\newblock A solution of {D}edekind's problem on the number of isotone {B}oolean
  functions.
\newblock {\em Journal fur die Reine und Angewandte Mathematik}, 1988(386):139
  – 144, 1988.

\bibitem{switch1_kohavi1970switching}
Zvi Kohavi.
\newblock {\em Switching and finite automata theory}.
\newblock {M}cGraw-Hill ({N}ew {Y}ork, {NY} [ua]), 1970.

\bibitem{Korshunov2003929}
Aleksej~D. Korshunov.
\newblock Monotone {B}oolean functions.
\newblock {\em Russian Mathematical Surveys}, 58(5):929 – 1001, 2003.

\bibitem{maitra1962cascaded}
Karuna~K Maitra.
\newblock Cascaded switching networks of two-input flexible cells.
\newblock {\em IRE Transactions on Electronic Computers}, (2):136--143, 1962.

\bibitem{switch4_matheson1971recognition}
William~S. Matheson.
\newblock Recognition of monotonic and unate cascade realizable functions using
  an informational model of switching circuits.
\newblock {\em IEEE Transactions on Computers}, 100(10):1214--1219, 1971.

\bibitem{switch2_mcnaughton1961unate}
Robert McNaughton.
\newblock Unate truth functions.
\newblock {\em IRE Transactions on Electronic Computers}, EC-10(1):1--6, 1961.

\bibitem{comp3_morizumi2014sensitivity}
Hiroki Morizumi.
\newblock Sensitivity, block sensitivity, and certificate complexity of unate
  functions and read-once functions.
\newblock In {\em IFIP International Conference on Theoretical Computer
  Science}, pages 104--110. Springer, 2014.

\bibitem{mukhopadhyay1969unate}
Amar Mukhopadhyay.
\newblock Unate cellular logic.
\newblock {\em IEEE Transactions on Computers}, 100(2):114--121, 1969.

\bibitem{pawelski2021number}
Bart\l{}omiej Pawelski.
\newblock On the number of inequivalent monotone {B}oolean functions of 8
  variables.
\newblock \url{https://arxiv.org/abs/2108.13997}, 2021.

\bibitem{pawelski2023number}
Bart\l{}omiej Pawelski.
\newblock On the number of inequivalent monotone {B}oolean functions of 9
  variables.
\newblock \url{https://arxiv.org/abs/2305.06346}, 2023.

\bibitem{switch6_pitchumani1988functional}
Vijay Pitchumani and Satish~S. Soman.
\newblock Functional test generation based on unate function theory.
\newblock {\em IEEE transactions on computers}, 37(6):756--760, 1988.

\bibitem{RT09}
Fred Roberts and Barry Tesman.
\newblock {\em Applied Combinatorics, 2nd edition}.
\newblock Chapman and Hall/CRC, 2009.

\bibitem{sasao1979}
Tsutomu Sasao and Kozo Kinoshita.
\newblock On the number of fanout-free functions and unate cascade functions.
\newblock {\em {IEEE} Transactions on Computers}, 28(1):66--72, 1979.

\bibitem{ListINtSequence}
Neil~J.A. Sloane.
\newblock The online encyclopedia of integer sequences.
\newblock \url{https://oeis.org/}, 1964.

\bibitem{STEPHEN201415}
Tamon Stephen and Timothy Yusun.
\newblock Counting inequivalent monotone {B}oolean functions.
\newblock {\em Discrete Applied Mathematics}, 167:15--24, 2014.

\bibitem{switch3_thayse1976logic}
Andr\'e Thayse and Jean~P. Deschamps.
\newblock Logic properties of unate and of symmetric discrete functions.
\newblock In {\em Proceedings of the sixth international symposium on
  Multiple-valued logic}, pages 79--87, 1976.

\bibitem{Ward}
Morgan Ward.
\newblock Note on the order of free distributive lattices.
\newblock {\em Bull. Amer. Math. Soc.}, 52:423, 1946.

\bibitem{Wiedemann19915}
Doug Wiedemann.
\newblock A computation of the eighth {D}edekind number.
\newblock {\em Order}, 8(1):5 – 6, 1991.

\bibitem{comp2_zwick19914n}
Uri Zwick.
\newblock A $4n$ lower bound on the combinational complexity of certain
  symmetric Boolean functions over the basis of unate dyadic boolean functions.
\newblock {\em SIAM Journal on Computing}, 20(3):499--505, 1991.

\end{thebibliography}

\end{document}